\documentclass[12pt]{amsart}

\usepackage{amsmath, amsthm, amscd, amsfonts, amssymb, graphicx, color}
\usepackage{latexsym}
\usepackage{amscd}
\usepackage{amsthm}
\usepackage{amssymb}
\usepackage{enumerate}
\usepackage{mathabx}
\usepackage{color}
\usepackage{soul}
\setstcolor{red}
\setul{}{0.4ex}

\theoremstyle{plain}
\newtheorem{theorem}{Theorem}
\newtheorem{corollary}[theorem]{Corollary}

\newtheorem{prop}[theorem]{Proposition}
\newtheorem{definition}[theorem]{Definition}
\theoremstyle{remark}
\newtheorem{remark}[theorem]{Remark}
\newtheorem{example}[theorem]{Example}

\begin{document}

\title{On slow decay of Peetre's $K$-functional}

\author{J.~M.~Almira$^*$,  P.~Fern\'{a}ndez-Mart\'{\i}nez}

\keywords{Lethargy results, Rate of convergence, K-functional.}
\thanks{$^*$ Corresponding author}

\maketitle
\begin{abstract}
We characterize when Peetre's $K$-functional slowly decays to zero    and we use this characterization  to demonstrate certain strict inclusions between real interpolation spaces.
\end{abstract}

\markboth{J.~M.~Almira, P.~Fern\'{a}ndez-Mart\'{\i}nez}{Slow decay of Peetre's $K$-functional}

\section{Introduction }
In this paper we consider a couple $(X,Y)$, where $(X,\|\cdot\|_X)$ is a quasi-Banach space and $Y\subset X$ is a quasi-semi-normed space, $(Y, \|\cdot\|_Y)$ which is continuously embedded into $X$.  Under these conditions, we also consider Peetre's  $K$-functional, which is defined by
\[
K(x,t,X,Y)=\inf_{y\in Y} \|x-y\|_X+t\|y\|_Y
\] 
This functional is obviously connected to the approximation properties of the elements of $X$ by elements of $Y$. Concretely, if $Y$ is dense in $X$, then $$K(x,0^+,X,Y)=\lim_{t\to 0^+}K(x,t,X,Y)=0$$ and the claim 
\[ 
K(x,t,X,Y)\geq c 
\]
is equivalent to the following property:
\[
\text{If } y\in Y \text{ satisfies } \|x-y\|_X<c, \text{ then } \|y\|_Y\geq \frac{c-\|x-y\|_X}{t}
\] 
In particular, if $\|x-y\|_X<\frac{c}{2}$ with $y\in Y$, then $\|y\|_Y\geq \frac{c-\|x-y\|_X}{t}\geq \frac{c}{2t}$, which diverges to infinity when $t$ tends to $0$. Let $A$ be a subset of $X$,  we define 
$$K(A,t,X,Y)=\sup_{x\in A}K(x,t,X,Y).$$ 
Thus, if $S(X)$ denotes the unit sphere of $X$, the condition
\begin{equation} \label{condicion}
K(S(X),t,X,Y)>c \text{ for all } t>0
\end{equation}
implies the existence of elements in $X$ of small norm (in particular, of norm $\|x\|_X=1$) which are badly approximable by elements of $Y$ with small norm in $Y$, which is a property that holds true in many cases. For example, if we set $X=C[a,b]$ and $Y=C^{(m)}[a,b]$ with $m\geq 1$, then there are oscillating functions of small uniform norm that cannot be uniformly approximated by functions of small $C^{(m)}$-semi-norm. 

The connection between the $K$-functional and approximation theory is, in fact, a strong one, and there are many results which connect the behavior of this functional to the properties of best approximation errors with respect to an approximation scheme $(X,\{A_n\})$. In particular, the so called Central Theorems in Approximation Theory lead to a proof that, when the approximation scheme $(X,\{A_n\})$ satisfies Jackson's and Bernstein's inequalities with respect to the space $Y\hookrightarrow X$, which are given by 
$$E(y,A_n)\leq Cn^{-r}\|y\|_Y \text{ for all }y \in Y\text{ and all }n\geq 0$$
and
$$\|a\|_Y\leq Cn^r \|a\|_X \text{ for all } a\in A_n,$$
respectively,  then the approximation spaces $A_{q}^{\alpha}(X,\{A_n\})=\{x\in X: \{2^{\alpha k}E(x,A_{2^k})\}\in \ell_q\}$  are completely characterized as interpolation spaces by the formula (see \cite{DvL})
\[
A_q^{\alpha}(X,\{A_n\})= (X,Y)_{\alpha/r,q} \text{ for all } 0<\alpha<r\text{ and  all } 0<q\leq\infty.
\]
Here $(X,Y)_{\theta,q}$ denotes the real interpolation space 
\[
(X,Y)_{\theta,q}=\{x\in X: \rho_{\theta,q}(x)=\| t^{-\left(\theta+\frac{1}{q}\right)}K(x,t,X,Y)\|_{L^q(0,\infty)}<\infty\},
\]
which, in case that $Y\hookrightarrow X$, can be renormalized with the following equivalent quasi-norm:
\[
\|x\|_{\theta,q}=\|\{2^{\theta k}K(x,\frac{1}{2^k},X,Y)\}_{k=0}^{\infty}\|_{\ell_q}.
\]
It follows that, when $(X,\{A_n\})$ satisfies Jackson's and Bernstein's inequalities with respect to a subspace $Y$, the rates of convergence to zero of the sequence of best approximation errors $E(x,A_n)$ and the sequence of evaluations of Peetre's $K$-functional $K(x,t,X,Y)$ at points $1/n$ seem to have similar roles and, in particular, serve to describe the very same subspaces of $X$. Thus, it  is an interesting question to know under which conditions on the couple $(X,Y)$, with $Y\hookrightarrow X$, the $K$-functional $K(\cdot,\cdot,X,Y)$ approaches to zero slowly, not only in presence of an approximation scheme but also for the general case. In this paper we characterize couples whose associated $K$-functional decays to zero slowly. As an application,  we use this characterization to demonstrate that some natural embeddings between real interpolation spaces are strict, a question that was posed by Lions in the decade of 1960 and was solved  by Janson, Nilsson, Peetre and Zafran in 1984 \cite{JNPZ} with some sophisticated tools \cite{Levy}. A complete solution, for the complex method of interpolation, was previously shown by   Stafney in 1970 \cite{S}. Bergh and  L\"{o}fstr\"{o}m  presented a partial solution for the real interpolation method in their famous book  \cite{BL}. Their approach was inspired by Stafney's work. In this paper we study the very same cases by means of less sophisticated tools.

In section 2 of this paper we characterize $K$-functionals which slowly decay to zero as those satisfying \eqref{condicion} and we also characterize them in terms of the approximation properties of $Y$ as a subspace of $X$. In section 3 we demonstrate that for any couple $(X,Y)$ of quasi-Banach spaces such that $Y\hookrightarrow X$, $Y\neq X$, the $K$-functional $K(\cdot,\cdot, X,Y)$ decays to zero slowly and, as a consequence, we prove that, under the very same hypotheses, all interpolation spaces $(X,Y)_{\theta,p}$ are strictly embedded between $X$ and $Y$ and that, if $\theta_1\neq \theta_2$, then $(X,Y)_{\theta_1,p}\neq (X,Y)_{\theta_2,q}$ for all $0<p,q\leq \infty$. Moreover, if $0<\theta<1$ and $0<p,q\leq \infty$ are such that $(X,Y)_{\theta,p}\neq (X,Y)_{\theta,q}$, then we show that the spaces $(X,Y)_{\theta,r}$ with $r\in [p,q]$ are pairwise distinct. The proofs we introduce for these cases are new, elementary, and very natural ones. Moreover, the results we get about $K$-functionals with slow decay are also new and, in our opinion, interesting by themselves.  In the last section of the paper, we show that, if $Y$ is densely embedded into $X$ and $Z\hookrightarrow X$ is a subspace compactly embedded into $X$, then there is a function $\varphi(t)$ satisfying 
$\lim_{t\to 0^+}\varphi(t)=0$ such that $K(z,t,X,Y)\leq \varphi(t)\|z\|_Z$ for all $z\in Z$. Thus, all elements of $Z$ share a common decay to zero in their $K$-functionals.

\section{Characterizations of slowly decaying $K$-functionals}

The importance of condition \eqref{condicion}  comes from the following result: 

\begin{theorem}\label{uno} The following are equivalent claims:
\begin{itemize}
\item[$(a)$]  $K(S(X),t,X,Y)>c$ for all  $t>0$ and a certain constant $c>0$.
\item[$(b)$] For every non-increasing sequences $\{\varepsilon_n\},\{t_n\}\in c_0$, there are elements $x\in X$ such that 
\[
K(x,t_n,X,Y)\neq \mathbf{O}(\varepsilon_n)
\]
\end{itemize}
\end{theorem}
\begin{proof} $(a)\Rightarrow (b)$.  Let us assume, on the contrary, that $K(x,t_n,X,Y)= \mathbf{O}(\varepsilon_n)$ for all $x\in X$ and certain sequences $\{\varepsilon_n\},\{t_n\}\in c_0$. This can be reformulated as $$X=\bigcup_{m=1}^{\infty} \Gamma_m,$$ where $$\Gamma_m=\{x\in X: K(x,t_n,X,Y)\leq m\varepsilon_n \text{ for all }n\in\mathbb{N}\}.$$
Now, $\Gamma_m$ is a closed subset of $X$ for all $m$ and Baire category theorem implies that $\Gamma_{m_0}$ has nonempty interior for some $m_0\in\mathbb{N}$. On the other hand, $\Gamma_m=-\Gamma_m$ since $K(x,t,X,Y)=K(-x,t,X,Y)$ for all $x\in X$ and $t\geq 0$. Furthermore, if $C>1$ is the quasi-norm constant of $X$ and $Y$, then
\[
\textbf{conv}(\Gamma_m)\subseteq \Gamma_{Cm},
\]
since, if $x,y\in \Gamma_m$ and $\lambda\in [0,1]$, then 
\begin{eqnarray*}
K(\lambda x +(1-\lambda)y ,t_n,X,Y) &\leq&  C(K(\lambda x ,t_n,X,Y)+ K( (1-\lambda)y ,t_n,X,Y))\\
&=&  C(\lambda  K(x ,t_n,X,Y)+(1-\lambda) K( y ,t_n,X,Y)) \\
&\leq&  C(\lambda  m\varepsilon_n+(1-\lambda) m\varepsilon_n) \\
&=&  C m\varepsilon_n 
\end{eqnarray*} 
Thus, if $B_X(x_0,r)=\{x\in X:\|x_0-x\|_X<r\}\subseteq \Gamma_{m_{0}}$, then $\textbf{conv}(B_X(x_0,r)\cup B_X(-x_0,r))\subseteq \Gamma_{Cm_0}$. In particular, 
\[
\frac{1}{2}(B_X(x_0,r)+B_X(-x_0,r))\subseteq \Gamma_{Cm_0}.
\]
Now, it is clear that 
\[
B_X(0,r)\subseteq \frac{1}{2}(B_X(x_0,r)+B_X(-x_0,r)),
\]
since, if $x\in B_X(0,r)$, then $\|-x_0-(x+x_0)\|=\|x_0-(x-x_0)\|=\|x\|_X=\leq r$, so that  $x-x_0\in B_X(x_0,r)$, $x+x_0\in B_X(x_0,r)$, and $x=\frac{1}{2}((x-x_0)+(x+x_0))$. Thus, 
\[
B_X(0,r)\subseteq  \Gamma_{Cm_0}.
\]
This means that, if $x\in X\setminus \{0\}$, then 
\[
K(\frac{rx}{\|x\|},t_n,X,Y)\leq Cm_0\varepsilon_n \text{ for all } n=1,2,\cdots
\]
Hence 
\[
K(x,t_n,X,Y)\leq \frac{\|x\|}{r}Cm_0\varepsilon_n \text{ for all } n=1,2,\cdots
\]
On the other hand, $(a)$ implies that, for each $n$, there is $x_n\in S(X)$ such that $K(x_n,t_n,X,Y)>c$, so that 
\[
c< K(x_n,t_n,X,Y)\leq  \frac{1}{r}Cm_0\varepsilon_n \text{ for all } n=1,2,\cdots,
\] 
which is impossible, since $\varepsilon_n$ converges to $0$ and $c>0$. This proves $(a)\Rightarrow (b)$. 

Let us demonstrate the other implication. Assume that $(a)$ does not hold. Then there are non-increasing sequences $\{t_n\}, \{c_n\}\in c_0$ such that 
\[
K(S(X),t_n,X,Y)\leq c_n \text{ for all } n\in\mathbb{N}
\]
In particular, if $x\in X$ is not the null vector, then
\[
K(\frac{x}{\|x\|_X},t_n,X,Y)\leq K(S(X),t_n,X,Y)\leq c_n \text{ for all } n\in\mathbb{N},
\]
so that 
\[
K(x,t_n,X,Y)=\|x\|_X K(\frac{x}{\|x\|_X},t_n,X,Y)\leq \|x\|_X c_n \text{ for all } n\in\mathbb{N},
\]
and $K(x,t_n,X,Y)=\mathbf{O}(c_n)$ for all $x\in X$. This proves $(b)\Rightarrow (a)$. 
\end{proof}

\begin{definition} We say that the $K$-functional $K(\cdot,\cdot, X,Y)$ slowly decays to zero if either condition $(a)$ or $(b)$ of Theorem \ref{uno} hold true. 
\end{definition}

\begin{remark} Note that, if $Y\subseteq X$ and the $K$-functional $K(\cdot,\cdot, X,Y)$ decays to zero slowly, then $(X,Y)_{\theta,q}$ is properly contained into $X$ for all $0< \theta< 1$ and all $0<q\leq \infty$. 
\end{remark}

Next proposition, whose proof is omitted, states that if the elements of $S(X)$ are not easily approximated by elements of $Y$,  they cannot  be easily approximated by elements of any other subspace $Y_{1} \hookrightarrow Y$. 
\begin{prop}  Assume that  $Y$ is densely embedded into $X$ and $K(S(X),t,X,Y)>c$ for all  $t>0$ and a certain constant $c>0$. Then:
\begin{itemize}
\item[(a)] If $Y_1\hookrightarrow Y$ is another quasi-semi-normed space continuously embedded into $Y$, then  $K(S(X),t,X,Y_1)>c$ for all $t>0$. 
\item[(b)] If $Y_1\hookrightarrow Y$ is another quasi-semi-normed space continuously embedded into $Y$, then 
for every every non-increasing sequences $\{\varepsilon_n\},\{t_n\}\in c_0$, there are elements $x\in X$ such that 
\[
K(x,t_n,X_1,Y_1)\neq \mathbf{O}(\varepsilon_n) .
\]
\end{itemize}
\end{prop}

\begin{remark} \label{pro1} Assume that $\|\cdot\|_{*,X}$ is an equivalent quasi-norm of $X$ and  $\|\cdot\|_{*,Y}$ is an equivalent quasi-semi-norm of $Y$, and let us denote by $$K^*(x,t,X,Y)=\inf_{y\in Y}\|x-y\|_{*,X}+t\|y\|_{*,Y}$$ the $K$-functional associated to the couple $(X,Y)$ when doted with the norms $\|\cdot\|_{*,X}$ and $\|\cdot\|_{*,Y}$, respectively. Then there are two constants $M,N>0$ such that $$M\cdot K^*(x,t,X,Y)\leq K(x,t,X,Y)\leq N\cdot K^*(x,t,X,Y)$$
In particular, if $S^*(X)=\{x\in X: \|x\|_{*,X}=1\}$, the following are equivalent claims:
\begin{itemize}
\item[$(a)$]  $K^*(S^*(X),t,X,Y)>c^*$ for all  $t>0$ and a certain constant $c^*>0$. 
\item[$(b)$] $K(S(X),t,X,Y)>c$ for all  $t>0$ and a certain constant $c>0$.
\end{itemize}
\end{remark}
Let us now characterize, in terms of the approximation properties of $Y$ as a subspace of $X$, the couples 
$(X,Y)$ whose associated $K$-functional slowly decays to zero:
 
\begin{theorem}\label{dos} The following are equivalent claims:
\begin{itemize}
\item[$(a)$]  $K(S(X),t,X,Y)>c$ for all  $t>0$ and a certain constant $c>0$.
\item[$(b)$]  For a certain constant $\delta>0$ there exist sequences $\{x_n\}\subset S(X)$ and $\{b_n\}\subset ]0,\infty[$, $\lim b_n= \infty$, such that, for every $n\in\mathbb{N}$,  
\[
\|x_n-y\|_X<\delta \text{ and } y \in Y \text{ imply that } \|y\|_Y\geq b_n
\]
\end{itemize}
\end{theorem}
\begin{proof}
Note that we can assume $0<c,\delta<1$ without loss of generality.
Assume $(a)$ and take, for each $t>0$, an element $x_t\in S(X)$ such that $K(x_t,t,X,Y)>c$ then we have already shown that if $\|x_t-y_t\|_X<\frac{c}{2}$ with $y_t\in Y$, then $\|y_t\|_Y\geq \frac{c-\|x_t-y_t\|_X}{t}\geq \frac{c}{2t}$. Thus, $(b)$ holds for $\delta=c/2$ and $b_n=\frac{c}{2t_n}$ for every decreasing sequence $\{t_n\}\in c_0$. This proves $(a)\Rightarrow (b)$.  
  
Let us now assume that $K(x_n,\frac{1}{b_n},X,Y)<\delta$ with $\{x_n\}$, $\delta$ and $\{b_n\}$ satisfying $(b)$. Then there is $y_n\in Y$ such that $\|x_n-y_n\|_X+\frac{1}{b_n}\|y_n\|_Y<\delta$, which leads to a contradiction since $\|x_n-y_n\|_X<\delta$ and $y_n\in Y$ imply $\|y_n\|_Y\geq b_n$, so that 
\[
1>\delta> \|x_n-y_n\|_X+\frac{1}{b_n}\|y_n\|_Y>1.
\]
It follows that $K(S(X),\frac{1}{b_n},X,Y)\geq \delta$ for all $n$ and $(a)$ follows with $c=\delta$ from the monotonicity of $K$ and the fact that $\lim_{n\to\infty}\frac{1}{b_n}=0$. 
\end{proof}

Theorem \ref{dos} reveals an easy way to confirm that Peetre's $K$-functional decays to zero slowly, since condition $(b)$ of this theorem is, in many cases, easy to verify. We include here just a few examples to show the way  this theorem can be used.

\begin{example} Set $X=C[a,b]$ with the uniform norm $\|f\|_{C[a,b]}=\sup_{t\in [a,b]}|f(t)|$ and $Y=C^{(1)}[a,b]$ with the seminorm $\|g\|_{C^{(1)}[a,b]}=\|g'\|_{C[a,b]}$. 
Take $\alpha_n=\frac{a+b}{2}-\frac{1}{2n}, \beta_n=\frac{a+b}{2}+\frac{1}{2n}$ and set
$$f_n(x)=\left\{ \begin{array}{llll}  -1 & x\in [a,\alpha_n] \\
\frac{2}{\beta_n-\alpha_n}x-\frac{\alpha_n+\beta_n}{\beta_n-\alpha_n} & x\in [\alpha_n,\beta_n]\\
1 & x\in [\beta_n,b]
\end{array}
\right.  $$
Then $\|f_n\|_{C[a,b]}=1$ for all $n$ and, if $\|f_n-g\|_{C[a,b]}<\frac{1}{2}$ with $g\in C^{(1)}[a,b]$ then $g(\alpha_n)<-1/2$ and $g(\beta_n)>1/2$ so that, there is a point $\rho_n\in[\alpha_n,\beta_n]\subseteq [a,b]$ such that 
$$g'(\rho_n)=\frac{g(\beta_n)-g(\alpha_n)}{\beta_n-\alpha_n}\geq \frac{1}{1/n}=n$$
and condition $(b)$ of Theorem \ref{dos} holds true with $\delta=\frac{1}{2}$ and $b_n=n$. Note that this example also works when we consider on $Y=C^{(1)}[a,b]$ the norm $\|g\|_1= \|g\|_{C[a,b]}+\|g'\|_{C[a,b]}$. 
\end{example}
\begin{example} Given $0<p<q\leq \infty$ we set $X=\ell_q$ with the standard norm (or quasi-norm) $\|\{a_n\}\|_{q}=\left(\sum_{n=0}^{\infty}|a_n|^q\right)^{\frac{1}{q}}$, for $q<\infty$, and $X=c_0$ with the supremum norm $\|\{a_n\}\|_{\infty}=\sup_{n}|a_n|$, if $q=\infty$. We also set $Y=\ell_p$ with the standard norm (or quasi-norm). Take $\mathbf{a}_n=\{a_{k}^n\}_{k=0}^{\infty}$, where $a_k^n=1$ for $k=0,1,\cdots, 2n-1$ and $a_k^n=0$ for $k\geq 2n$. If $\|\mathbf{a}_n-\mathbf{b}\|_{\infty}<\frac{1}{2}$ (analogously if  $\|\mathbf{a}_n-\mathbf{b}\|_{q}<\frac{1}{2}$), with $\mathbf{b}=\{b_k\}\in\ell_p$, then $b_k>\frac{1}{2}$ for all $k\in \{0,1,\cdots,2n-1\}$, which implies that $\|\mathbf{b}\|_p\geq 2^{-1+\frac{1}{p}}n^{\frac{1}{p}}$. Thus, condition $(b)$ of Theorem \ref{dos} holds true in all  these cases with $\delta=\frac{1}{2}$ and $b_n=2^{-1+\frac{1}{p}}n^{\frac{1}{p}}$
\end{example}

\begin{example} If we do not impose $Y\hookrightarrow  X$ but we maintain $Y\subseteq X$ (i.e., we do not require continuity of the inclusion) then it is easy to find a couple $(X,Y)$ whose $K$-functional does not decay to zero slowly. Concretely, we can take $X=\ell_1$ with its natural norm, $\|\cdot\|_1$, and $Y=c_{00}=\{\text{ finite sequences }\}$ with the norm of supremum , $\|\cdot\|_{\infty}$.  Obviously $Y\subset X$. Take $(a_n)\in S(\ell_1)$ and take $\delta>0$ arbitrary. Then there is a number $N_0\in\mathbb{N}$ such that $\sum_{n=N_0+1}^{\infty}|a_n|<\delta$. Take $(b_n)\in c_{00}$ given by $b_k=a_k$ for $k\leq N_0$ and $b_k=0$ otherwise. Then 
\[
K((a_n),t,\ell_1,c_{00})\leq \|(a_n)-(b_n)\|_1+t\|(b_n)\|_{\infty}\leq \delta+t,
\]
and this holds for every $\delta>0$, so that 
\[
K((a_n),t,\ell_1,c_{00})\leq t
\]
and this $K$-functional does not decay to zero slowly.

\end{example}

\section{The Quasi-Banach setting}

The following result shows that, under very mild conditions on the couple $(X,Y)$, the $K$-functional 
$K(\cdot,\cdot, X,Y)$ slowly decays to zero.

\begin{theorem} \label{fundamental}
Assume that $(X,\|\cdot\|_{X})$ and $(Y,\|\cdot\|_Y)$ are quasi-Banach spaces. Then
\begin{itemize}
\item[$(a)$] If $Y\hookrightarrow X$, $Y\neq X$ (i.e., $Y$ is properly embedded into $X$). Then $K(S(X),t,X,Y)>c$ for all $t>0$ and a certain constant $c>0$. 
\item[$(b)$] If $X$ and $Y$ are both $p$-normed spaces, $Y\hookrightarrow X$, then either $K(S(X),t,X,Y)=1$ for all $t>0$, or $X=Y$. 
\end{itemize}
\end{theorem}
\begin{proof}
It follows from Aoki-Rolewicz Theorem (see \cite[Theorem 2.1.1]{DvL} or \cite[pages 7-8]{KPR}) that $X$ and $Y$ can be renormed with equivalent quasi-norms in such a way that they become $p$-normed spaces for a certain $p\in ]0,1]$ (indeed, Aoki-Rolewicz gives two equivalent quasi-norms $\|\cdot\|_{*,X}$ and $\|\cdot\|_{*,Y}$, respectively and two numbers $p_X,p_Y\in ]0,1]$ such that 
$\|x_1+x_2\|_{*,X}^{p_X}\leq \|x_1\|_{*,X}^{p_X}+\|x_2\|_{*,X}^{p_X}$ and $\|y_1+y_2\|_{*,Y}^{p_Y}\leq \|y_1\|_{*,Y}^{p_Y}+\|y_2\|_{*,Y}^{p_Y}$ for all $x_1,x_2\in X$ and all $y_1,y_2\in Y$. Thus, if we set $p=\min\{p_X,p_Y\}$ then  
$\|x_1+x_2\|_{*,X}^{p}\leq \|x_1\|_{*,X}^{p}+\|x_2\|_{*,X}^{p}$ and $\|y_1+y_2\|_{*,Y}^{p}\leq \|y_1\|_{*,Y}^{p}+\|y_2\|_{*,Y}^{p}$ for all $x_1,x_2\in X$ and all $y_1,y_2\in Y$ and both spaces are $p$-normed with the very same $p$. It follows from this and from Proposition \ref{pro1} that we only need to demonstrate part $(b)$ of the Theorem. 
  
Thus, we assume that $X$, $Y$ are quasi-Banach $p$-normed spaces and  $Y\hookrightarrow X$ and we want to demonstrate that, if $K(S(X),t,X,Y)<1$ for a certain $t>0$, then $X=Y$. 
Assume, on the contrary, that  $K(S(X),t_0,X,Y)=c<1$ and $X\neq Y$. Take $\rho\in ]0,1[$ such that $c<\rho^{1/p}<1$. Then $K(\frac{x}{\|x\|_X},t_0,X,Y)<\rho^{1/p}$ for every $x\in X$, $x\neq 0$. Hence every element $x$ of $X$ which is different from zero satisfies $$K(x,t_0,X,Y)<\rho^{1/p}\|x\|_X,$$
which implies that $\|x-y_0\|_X+t_0\|y_0\|_Y<\rho^{1/p}\|x\|_X$ for certain $y_0\in Y$. Thus, if we set $x_0=x-y_0$, we have that 
\begin{equation} \label{argumento}
\left\{\begin{array}{llllll} x =x_0+y_0 \text{ with } x_0\in X, \text{ and } y_0\in Y \\
\|x_0\|_X< \rho^{1/p}\|x\|_X \\
\|y_0\|_Y< t_0^{-1}\rho^{1/p}\|x\|_X \\
\end{array} \right.
\end{equation}
Let us take $x\in X\setminus Y$ and apply \eqref{argumento} to this concrete element. We can repeat the argument just applying it to $x_0$ (if $x_0=0$ then $x=y_0\in Y$, which contradicts our assumption). Hence, there are elements $x_1\in X$ and $y_1\in Y$ such that 
\begin{eqnarray*} 
\left\{\begin{array}{llllll} x_0 =x_1+y_1 \text{ with } x_1\in X, \text{ and } y_1\in Y \\
\|x_1\|_X< \rho^{1/p}\|x_0\|_X < (\rho^{1/p})^2\|x\|_X \\
\|y_1\|_Y< t_0^{-1}(\rho^{1/p})^2\|x\|_X \\
\end{array} \right.
\end{eqnarray*}
Moreover, $x=x_0+y_0=x_1+y_1+y_0$. Again $x_1\neq 0$ since $x\not\in Y$. We can repeat the argument $m$ times to get a decomposition $x=x_m+y_m+\cdots+y_0$ with $x_m\in X$, $x_m\neq 0$, $y_k\in Y$ for all $0\leq k\leq m$ and 
\begin{equation*} 
\left\{\begin{array}{llllll} \|x_m\|_X< (\rho^{1/p})^{m+1}\|x\|_X \\
\|y_k\|_Y< t_0^{-1}(\rho^{1/p})^{k+1}\|x\|_X \text{ for all } 0\leq k\leq m .\\
\end{array} \right. 
\end{equation*}
Let us set $z_m=x-x_m=y_0+\cdots+y_m$. Then 
\[
\|x-z_m\|_X=\|x_m\|_X<(\rho^{1/p})^{m+1}\|x\|_X \to 0 \text{ for } m\to\infty.
\]
and $x$ is the limit of $z_m$ in the norm of $X$.
On the other hand, if $n>m$, then 
\[
\|z_n-z_m\|_Y^p=\|y_{m+1}+\cdots+y_n\|_{Y}^p\leq \sum_{k=m+1}^n \|y_k\|_Y^p \leq t_0^{-p}\|x\|_X^p\sum_{k=m+1}^n \rho^{k+1},
\]
which converges to $0$ for $n,m\to\infty$. Hence $\{z_m\}$ is a Cauchy sequence in $Y$ and its limit belongs to $Y$ since $Y$ is topologically complete. This implies $x\in Y$, which contradicts our assumptions. Thus, we have demonstrated, for $p$-normed quasi-Banach spaces $X$ and $Y$ satisfying $Y\hookrightarrow X$, that  if $K(S(X),t_0,X,Y)<1$ for a certain $t_0>0$, then $X=Y$. In particular, if $X\neq Y$ then $K(S(X),t,X,Y)=1$ for all $t>0$. 
\end{proof}

\begin{corollary} Let $(X,Y)$ be an ordered couple of quasi-Banach spaces, $Y\hookrightarrow X$, $Y\neq X$. Then there exist a constant  
$\delta>0$; a sequence $\{b_n\}\subset [0,\infty)$ with $ (b_n) \to \infty$ as $n \to \infty$ and a sequence $\{x_n\}\subseteq S(X)$ such that for all $n\in\mathbb{N}$ and $y\in Y$ 
 $$\|x_n-y\|_X<\delta \Rightarrow \|y\|_Y>b_n. $$  
\end{corollary}
\begin{proof} The result is a direct application of Theorems \ref{fundamental} and \ref{dos}. 
\end{proof}

Theorem \ref{fundamental} has some interesting consequences on the theory of interpolation spaces. Concretely, if $X$, $Y$ are quasi-Banach spaces and $Y\hookrightarrow X$, $Y\neq X$, then all interpolation spaces $(X,Y)_{\theta,q}$ are strictly embedded into $X$ and, if $Y$ is not closed in $X$, they strictly contain $Y$. Moreover, the inclusions we get by applying the reiteration theorem are also strict. This is stated in the following theorems:

\begin{theorem}\label{strict} Let $(X,Y)$ be a couple of quasi-Banach spaces, $Y\hookrightarrow X$, $Y\neq X$ and $Y$ not closed in $X$. Then 
$$Y\hookrightarrow (X,Y)_{\theta,q}\hookrightarrow X, \text{ with strict inclusions, for } 0<\theta<1\text{ and }0<q\leq \infty.$$ 
\end{theorem}
\begin{proof} 
The fact that the inclusion $(X,Y)_{\theta,q}\hookrightarrow X$ is strict follows directly from Theorems \ref{uno} and  \ref{fundamental}, since they allow us to claim that there are elements $x\in X$ such that $K(x,\frac{1}{2^k},X,Y)$ goes to zero as slow as we want, which implies that we can find $x\in X$ such that  $\|x\|_{(X,Y)_{\theta,q}}=\infty$. The inclusion $Y\hookrightarrow (X,Y)_{\theta,q}$ is also strict. This should be well known, but we include here a proof for the sake of completeness. We assume $q<\infty$, since the case $q=\infty$ admits a similar proof. The fact that $Y\hookrightarrow (X,Y)_{\theta,q}$ as soon as $\theta<1$ is a direct computation, based on the fact that, for $y\in Y$, $K(y,t,X,Y)\leq t\|y\|_Y$ for all $t>0$. 
As $Y\hookrightarrow X$ with strict inclusion, we know that $\|y\|_X\leq M\|y\|_Y$ for a certain $M>0$ and all $y\in Y$ and, moreover, there exists a sequence $\{y_n\}_{n=0}^{\infty}\subseteq S(Y)$ such that $\|y_n\|_X$ converges to $0$ when $n$ goes to infinity. In particular, for every $N_0\in\mathbb{N}$ there exists $y_{N_0}\in S(Y)$ such that $\|y_{N_0}\|_X\leq 2^{-N_0}$. This obviously implies that 
\[
K(y_{N_0},t,X,Y)\leq \left\{\begin{array}{lllll} 
\frac{1}{2^{N_0}} & \text{ if } t\geq \frac{1}{2^{N_0}} \\
t & \text{ if } t< \frac{1}{2^{N_0}} \\
\end{array} \right.
\]
Hence
\begin{eqnarray*}
\|y_{N_0}\|_{\theta,q} & = & \left[ \sum_{k=0}^{\infty} 2^{k\theta q} K(y_{N_0}, \frac{1}{2^k},X,Y)^q\right]^{\frac{1}{q}} \\
 & \leq & \left[ \sum_{k=0}^{N_0} 2^{k\theta q} 2^{-N_0q}+ \sum_{k=N_0+1}^{\infty} 2^{(\theta-1)kq} \right]^{\frac{1}{q}}\\
 &=&  \left[ \frac{2^{\theta q (N_0+1)}-1}{2^{\theta q}-1}2^{-N_0q}+ \frac{2^{(\theta-1)q(N_0+1)}}{1-2^{(\theta-1)q}} \right]^{\frac{1}{q}}\\
 &\leq&  \left[ 2^{(\theta-1)qN_0}\frac{2^{\theta q}}{2^{\theta q}-1} + \frac{2^{(\theta-1)q(N_0+1)}}{1-2^{(\theta-1)q}} \right]^{\frac{1}{q}},\\ 
\end{eqnarray*}
which converges to $0$ when $N_0$ goes to infinity. This demonstrates that the norms $\| \cdot\| _Y$ and $\|\cdot \|_{\theta,q}$ are not equivalent on $Y$, which implies that $Y\neq (X,Y)_{\theta,q}$. 
\end{proof}
\begin{theorem}\label{interpol} Let $(X,Y)$ be a couple of quasi-Banach spaces, $Y\hookrightarrow X$, $Y\neq X$. Then 
\begin{itemize}
\item[$(a)$] Assume that $0<\theta_0, \theta_1<1$, $\theta_0\neq \theta_1$, and $0<p,q\leq \infty$. Then  
 $$(X,Y)_{\theta_0,p}\neq  (X,Y)_{\theta_1,q}$$
\item[$(b)$] Let $\theta\in ]0,1[$ and assume that $0<p,q\leq \infty$ are such that $(X,Y)_{\theta,p}\neq  (X,Y)_{\theta,q}$ and that $r_1,r_2\in [p,q]$, $r_1\neq r_2$. Then $$(X,Y)_{\theta,r_1}\neq (X,Y)_{\theta,r_2}$$
\item[$(c)$] If $0<\theta<1$ and $0<p\leq \infty$,  then $(X,Y)_{\theta,p}$ is an infinite-codimensional subspace of $X$ and $Y$ is an infinite-codimensional subspace of $(X,Y)_{\theta,p}$. 
\end{itemize}
\end{theorem}
\begin{proof}
Part $(a)$ follows from Theorem \ref{fundamental} and the fact that $X$ is a space of class $\mathcal{C}(0,(X,Y))$, which implies that we can use the following reiteration formula:
\begin{equation}\label{reiteX}
(X,(X,Y)_{\theta,p})_{\alpha,q}=(X,Y)_{\alpha\theta,q}
\end{equation}
Indeed, if $0<\theta_1<\theta_0<1$, then $\theta_1=\alpha\theta_0$ for a certain $\alpha\in]0,1[$, which implies that 
\begin{equation}\label{reit2}
(X,Y)_{\theta_1,q}=(X,Y)_{\alpha\theta_0,q}=(X,(X,Y)_{\theta_0,p})_{\alpha,q}
\end{equation}
Thus, Theorem \ref{fundamental} guarantees that $(X,Y)_{\theta_0,p}$ is (a quasi-Banach space) strictly embedded into $X$, so that applying the very same result to the couple $(X,(X,Y)_{\theta_0,p})$, and formula \eqref{reit2}, we get that $(X,Y)_{\theta_1,q}$ is strictly embedded between $(X,Y)_{\theta_0,p}$ and $X$. In particular, $$(X,Y)_{\theta_1,q}\neq (X,Y)_{\theta_0,p} .$$
Let us assume $p>q>0$. To demonstrate part $(b)$ we use the reiteration formula
\begin{equation}\label{reit3}
((X,Y)_{\theta,p},(X,Y)_{\theta,q})_{\eta,r}=(X,Y)_{\theta, r},\text{ where } \frac{1}{r}=\frac{1-\eta}{p}+\frac{\eta}{q},
\end{equation}
which, in conjunction with Theorem \ref{fundamental}, implies that, if $(X,Y)_{\theta,p}\neq  (X,Y)_{\theta,q}$  then all interpolation spaces $(X,Y)_{\theta, r}$ with $r\in ]p,q[$ are strictly embedded between the spaces $(X,Y)_{\theta,p}$ and  $(X,Y)_{\theta,q}$ (note that, if $p>q$, then $(X,Y)_{\theta,q}\hookrightarrow (X,Y)_{\theta,p}$ since $\ell_q\hookrightarrow \ell_p$). This is so because   $\frac{1}{r}=\frac{1-\eta}{p}+\frac{\eta}{q}$ with $0\leq\eta\leq1$ is another way to claim that $\frac{1}{r} \in [\frac{1}{p},\frac{1}{q}]$ which is equivalent to $r\in [q,p]$. Thus, if $r_1,r_2\in ]q,p[$, $r_1<r_2$, then $(X,Y)_{\theta, r_1}$ is strictly embedded in $(X,Y)_{\theta,p}$ and  henceforth $(X,Y)_{\theta,r_2}$ is strictly embedded between the spaces $(X,Y)_{\theta,p}$ and  $(X,Y)_{\theta,r_1}$, since $r_1<r_2<p$. In particular,  $(X,Y)_{\theta, r_1}\neq  (X,Y)_{\theta, r_2}$.

Part $(c)$ follows directly from $(a)$ and Theorem \ref{strict}.

\end{proof}
\begin{remark}
 Observe that statement $(b)$ of Theorem \ref{interpol}  also follows from part  $(a)$ and formula \eqref{reit3}, since choosing two distinct $r_1,r_2$ in $]p,q[$ is the same as choosing two distinct $\eta_1,\eta_2$ in the lefthand side of  formula \eqref{reit3}. 
\end{remark}

In the case we consider a couple $(X,Y)$ where $(X,\|\cdot\|_X)$ is a quasi-Banach space and $Y\subset X$ is a quasi-semi-normed space $(Y, \|\cdot\|_Y)$ which is continuously included into $X$, we can no longer guarantee  the slow decay of the associated $K$-functional. For example, in \cite{Ditz} it was proved that, if $0<p<1$ and $r\in\mathbb{N}$ is a positive natural number, then $K(f,t,L_p,W_p^r)=0$ for all $f\in L_p[0,1]$, when we deal the Sobolev space $W_p^r$ with the quasi-seminorm $\|g\|_{p,r}=\|g^{(r)}\|_{L_p[0,1]}$, so that
\[
K(f,t,L_p,W_p^r)=\inf_{g^{(r-1)} \text{ is absolutely continuous}}\|f-g\|_{L_p[0,1]}+t\|g^{(r)}\|_{L_p[0,1]}.
\]
In spite of this, we can use  Theorem \ref{fundamental} to prove the following:
\begin{theorem} Assume that  $(X,Y)$ is a couple, where $(X,\|\cdot\|_X)$ is a quasi-Banach space and $Y\subset X$ is a quasi-semi-normed space $(Y, \|\cdot\|_Y)$ which is continuously embedded into $X$. If the interpolation space $(X,Y)_{\theta_0,p_0}$ is strictly embedded into $X$ for some choice of $0<\theta_0<1$ and $0<p_0\leq \infty$, then $$K(S(X),t,X,Y)>c \text{ for all } t>0 \text{ and a certain } c>0.$$
In particular, the strict inclusion of one interpolation space $(X,Y)_{\theta_0,p_0}$ into $X$ implies that all interpolation spaces  
$(X,Y)_{\theta,p}$ are strictly embedded into $X$, where $0<\theta<1$ and $0<p\leq \infty$.
\end{theorem}

\begin{proof} A direct application of Theorem \ref{fundamental} to the couple of quasi-Banach spaces 
$(X,(X,Y)_{\theta_0,p_0})$ leads to 
\[
K(S(X),t,X,(X,Y)_{\theta_0,p_0})>c \text{ for all } t>0 \text{ and a certain } c>0.
\]
Thus, for each $t,\varepsilon >0$ there is $x_t\in S(X)$ such that $K(x_t,t,X,(X,Y)_{\theta_0,p_0})>(1-\varepsilon)c$. On the other hand, we have that $Y\hookrightarrow (X,Y)_{\theta_0,p_0}$, which means that $\|y\|_{\theta_0,p_0}\leq M\|y\|_Y$ for all $y\in Y$ and a certain constant $M=M(\theta_0,p_0)$. Hence
\begin{eqnarray*}
K(x_t,t M, X,Y) & =& \inf_{y\in Y} \|x_t-y\|_X+t  M \|y\|_{Y} \\
&\geq &  \inf_{y\in Y} \|x_t-y\|_X+t   \|y\|_{\theta_0,p_0} \\
&\geq &  \inf_{z\in (X,Y)_{\theta_0,p_0}} \|x_t-z\|_X+t   \|z\|_{\theta_0,p_0}  \text{ (since } Y\subseteq (X,Y)_{\theta_0,p_0} \text{ )}\\
&=& K(x_t,t,X,(X,Y)_{\theta_0,p_0}) >(1-\varepsilon)c
\end{eqnarray*}
Hence $$K(x_{\frac{t}{M}},t, X,Y) = K(x_{\frac{t}{M}},\frac{t}{M} M, X,Y)>(1-\varepsilon)c,$$
which implies that $K(S(X),t,X,Y)\geq c>0$ for all $t>0$, since $\varepsilon,t>0$  are arbitrary. This ends the proof. 
\end{proof}

\section{Spaces whose elements share a common decay in their $K$-functionals respect to a given couple}

Another interesting question is to know under which conditions on a subspace $Z$ of $X$ we have that, for each decreasing sequence $\{t_n\}$ all elements  $z\in Z$ share a common rate of convergence to zero for the sequence $\{K(z,t_n,X,Y)\}$. We give a sufficient condition in the following result:

\begin{theorem} \label{tres} Assume that $Y$ is a continuously embedded dense subspace of $X$. If $Z$ is a compactly embedded subspace of $X$, and $\{t_n\}\in c_0$ is non-increasing, then there exists a non increasing sequence $\{\varepsilon_n\}\in c_0$ such that $$K(z,t_n,X,Y)\leq \|z\|_Z\varepsilon_n \text{ for all } n\in\mathbb{N} \text{ and all }z\in Z$$  
In fact, if we set $\varphi(t)= K(S(Z),t,X,Y)$, then $\lim_{t\to 0^+}\varphi(t)=0$ and $K(z, t, X,Y) \leq \|z\|_Z \varphi(t)$ for all $z\in Z$. 
\end{theorem}
\begin{proof}
Take $\varepsilon_n=K(S(Z),t_n,X,Y)$, where $S(Z)$ denotes the unit sphere in $Z$. Then 
\[
K(z,t_n,X,Y)=\|z\|_Z K(\frac{z}{\|z\|_Z},t_n,X,Y)\leq \|z\|_{Z}\varepsilon_n \text{ for all } n\in\mathbb{N},
\]
so that the result is proved as soon as we demonstrate that $\{\varepsilon_n\}\in c_0$. Now, the sequence $\{\varepsilon_n\}$ is non-increasing because of the monotony of $K$. Thus, if this sequence does not converge to $0$, is uniformly bounded by a positive constant $c>0$. In particular, there is a sequence $\{z_n\}\subseteq S(Z)$ such that $K(z_n,t_nX,Y)\geq c$ for all $n$. The compactness of the imbedding $Z\to X$ implies that there is a subsequence 
$\{z_{n_k}\}_{k=1}^{\infty}$ and a point $z\in X$ such that $\lim_{k\to\infty}\|z_{n_k}-z\|_X=0$. Hence
\[
0<c\leq K(z_{n_k},t_{n_k},X,Y)\leq C[K(z_{n_k}-z,t_{n_k},X,Y)+K(z,t_{n_k},X,Y)]\to 0,
\]
which is impossible. 
Thus, we have demonstrated that, if $\varphi(t)= K(S(Z),t,X,Y)$, then $\lim_{t\to 0^+}\varphi(t)=0$ and $K(z, t, X,Y) \leq \|z\|_Z \varphi(t)$ for all $z\in Z$. 
\end{proof}

\begin{example} If $(X,Y)$ is a couple and $Y\hookrightarrow Z\hookrightarrow X$ with $Z$ compactly embedded into $X$, then we can define the space $(Z^*,\|\cdot\|_{*})= (Z,\|\cdot \|_X)$ which results from considering on $Z$ the norm of $X$ and, thanks to Theorem \ref{tres},  the associated $K$-functional will satisfy: 
\[
K(z,t,Z^*,Y)= \inf_{y\in Y}\|z-y\|_{*}+t\|y\|_Y =  \inf_{y\in Y}\|z-y\|_{X}+t\|y\|_Y \leq \|z\|_Z \varphi(t),  
\]
where $\varphi(t)=K(S(Z),t,X,Y)$. Thus, $K(z,t,Z^*,Y)$ does not decay to zero slowly. 
\end{example}

\bigskip

\noindent J. M.~ALMIRA\\
Departamento de Ingenier\'{\i}a y Tecnolog\'{\i}a de Computadores,  Universidad de Murcia. \\
 30100 Murcia, SPAIN\\
e-mail: \texttt{jmalmira@um.es}\\[2ex]

\noindent P.~FERN\'{A}NDEZ-MART\'{I}NEZ\\
Departamento de Matem\'{a}ticas,  Universidad de Murcia. \\
30100 Murcia, SPAIN\\
e-mail: \texttt{pedrofdz@um.es}\\[2ex]


\medskip 

\begin{thebibliography}{99}
\bibitem{BL} {\sc J. Bergh and J.  L\"{o}fstr\"{o}m,}  Interpolation spaces. An introduction,  Springer, Berlin, 1976.
\bibitem{DvL} {\sc R.A. DeVore, G.G. Lorentz, } Constructive Approximation, Springer, 1993.
\bibitem{Ditz} {\sc Z. Ditzian, V.H. Hristov, K.G. Ivanov, } Moduli of smoothness and $K$-functionals in $L_p$, $0<p<1$, {\it Constructive Approx. } {\bf 11}  (1995) 67--83. 
\bibitem{JNPZ} {\sc S. Janson, P. Nilsson, J. Peetre, M. Zafran, } Notes on Wolff's note on interpolation spaces, Proc. London Math. Soc. (3), 48 (1984) 283-299.
\bibitem{KPR} {\sc N.J. Kalton, N.T. Peck, J.W. Roberts, } An F-Space Sampler, in: London Math. Soc. Lecture Note Series, vol. 89, Cambridge University Press, 1984.
\bibitem{Levy} {\sc M. Levy,}  L'espace d'interpolation reel $(A_0,A_1)_{\theta,p}$ contient $\ell^p$, C. R. Acad. Sci. Paris, \textbf{289} (1979) 675--677. 
\bibitem{S} {\sc J. Stafney, } Analytic interpolation of certain multiplier spaces, Pacific J. Math., \textbf{32} (1970) 241-248.
\end{thebibliography}
\end{document}